%% file: main.tex
\documentclass[12pt]{amsart}

\input{preamble.tex}

\usepackage{MnSymbol}
\usepackage{tikz,tikz-3dplot}

\def\Manifolds{\mathbf{Manifolds}}
\DeclareMathOperator{\Map}{Map}

\DeclareMathOperator{\SymMonCat}{\mathrm{\mathbf{SymMonCat}}}
\DeclareMathOperator{\ManCat}{\mathrm{\mathbf{Manifolds}}}

\begin{document}

\title{A Bicategory Approach to Differential Cohomology}
\author{Markus Upmeier}
\date{\today}
\address{Mathematisches Institut\\
Georg-August Universität G\"ottingen\\
Bunsenstraße 3-5\\
D-37073 G\"ottingen}
\email{upmeier@uni-math.gwdg.de}
\maketitle

\begin{abstract}
A bicategory approach to differential cohomology is presented. Based on the axioms of Bunke-Schick, a symmetric monoidal groupoid is associated to differential refinements of cohomology theories. It is proven that such differential refinements are unique up to equivalence of the corresponding symmetric monoidal groupoids and the existing uniqueness results for rationally-even theories are interpreted in this framework.

Moreover we show how the bicategory formalism may be used to give a simple construction of a differential refinement for any generalized cohomology theory, based on a refinement of the Chern character to a strict transformation of bicategories.
\end{abstract}

\section{Introduction}

Differential cohomology combines ideas from gauge theory (connections, curvature) and stable homotopy theory. For example, Chern-Weil theory combines gauge theory and ordinary cohomology and early interest in differential cohomology \cite{cheeger_simons} arose from the fact that the Chern-Weil homomorphism lifts to (ordinary) differential cohomology. In many cases this lift carries more information than the Chern-Weil form and the characteristic class together, the classical example being that the holonomy of a flat bundle may be non-zero, even though the Euler form vanishes.

In a similar way, differential $K$-theory makes a natural appearance in the context of index theory. The classical Atiyah-Singer Index Theorem for families asserts the equality of the analytic index of a family of Dirac operators with a cohomological quantity, involving the Chern character of the symbol and an $\hat{A}$-class. In the presence of `additional geometry' one may pick canonical Chern-Weil representative differential forms. The differential Index Theorem \cite{bunke_schick_k}, \cite{lott} asserts that actually more is true: one may lift both sides of the equation (the diff.~analytic and diff.~topological index) to differential cohomology and these refined indices coincide.

Generalized differential cohomology theories were constructed in \cite{hopkins_singer} using simplicial homotopy theory and applied to the cohomology theory
called `Anderson dual of the sphere' to define a certain invariant. Other constructions of differential extensions for specific cohomology theories (K-theory, Landweber-exact cohomology theories) had also been known, raising the question of uniqueness for differential refinements. This question was partially answered by \cite{bunke_schick_uniqueness} in the
case of rationally even cohomology theories.

The purpose of this article is twofold: First, to illustrate that uniqueness of differential refinements in the general case holds as well, but in a weaker sense: a differential refinement is viewed as a functor to monoidal groupoids which are proven to be unique up to monoidal equivalence. Secondly, we give a simple construction of generalized differential cohomology theories in terms of monoidal categories. Technically, this is based on a refinement of the Chern character transformation in generalized cohomology to a monoidal functor.\\

The paper is organized as follows: in the remaining part of this section, we shall recall the axiomatic definition of differential cohomology theories. The mentioned uniqueness result is proven in section 2. In section 3 we discuss the refinement of the Chern character transformation, referring the more technical points to the appendix (section 5). This refinement is used in section 4 to construct differential cohomology theories.\\

For the convenience of the reader, we recall here the axiomatic definition of a differential extension $(\hat{E}^*, R, I, a)$ of a generalized cohomology theory $E^*$ in the sense of \cite{bunke_schick_uniqueness}:
Let $V:=E^*(pt)\otimes \bbR$ (a graded vector space) and let $\mathbf{Manifolds}$ denote the category of smooth manifolds.
A \emph{differential extension} $\hat{E}$ of $E$ is a family of contravariant functors $$\hat{E}^*: \mathbf{Manifolds}\longrightarrow \mathbf{Abel}\hspace{1cm} (*\in \bbZ)$$
into the category of Abelian groups together with natural transformations
\begin{align*}
\Omega^{*-1}(-;V)/im(d) & \overset{a}\Longrightarrow \hat{E}^*,\\
\hat{E}^* & \overset{R}\Longrightarrow \Omega^*(-;V),\\
\hat{E}^*	& \overset{I}\Longrightarrow E^*(-),
\end{align*}
such that the following diagram commutes for any manifold $M$ and has an exact upper horizontal line
$$\xymatrix{
E^{*-1}(M)\ar[r]^-{ch}	&	\Omega^{*-1}(M;V)/im(d)\ar[r]^-a\ar[rd]_d	&	\hat{E}^*(M)\ar[r]^-I\ar[d]_R	&	E^*(M)\ar[r]\ar[d]^{ch}	&	0\\
&&\Omega_{cl}^*(M;V)\ar[r]	&	H_{sing}^*(M;V)
}$$
Here we have used differential forms and singular cohomology with coefficients in a graded vector space. Of course, an \emph{isomorphism of differential cohomology theories} is defined to be a natural transformation which is compatible with the maps $a,R,I$.\\

\thanks{\textbf{Acknowledgements.} I would like to thank Ralf Meyer for many fruitful discussions out of which this paper grew. Also, I would like to thank my advisor Thomas Schick for his constant support during my PhD research.}

\section{Uniqueness of Differential Extensions}
Given a differential refinement $(\hat{E}^*, R, I, a)$ of a cohomology theory, the transformation $a: \Omega^{*-1}(M;V)_d\rightarrow \hat{E}^*(M)$ may be used to define a groupoid
\[
	\tilde{E}^*(M)=\left\{\begin{tabular}{p{27ex}l}
	\textbf{objects:} 										&$x\in \hat{E}^*(M),$\\[2ex]
	\textbf{morphisms} from\newline $x\in \hat{E}^*(M)$ to $y\in \hat{E}^*(M)$:	& $\eta \in \Omega^{*-1}(M;V)_d$ with $a(\eta) = y-x$.
	\end{tabular}\right.
\]
Here, $\Omega^{*-1}(M;V)_d$ denotes the quotient of $\Omega^{*-1}(M;V)$ by $\im(d)$.
Using addition, $\tilde{E}^*(M)$ becomes a strict symmetric monoidal groupoid.
For varying $M$, these may be organized into a functor
\[
	\tilde{E}^*: \Manifolds \longrightarrow \SymMonCat_\mathrm{str}.
\]
In a similar way, the maps $d: \Omega^{*-1}(M;V)_d \longrightarrow \Omega^*(M;V)$ and $\delta: C^{*-1}(M;V)_\delta \rightarrow C^*(M;V)$ may be used to construct functors
\[
	\tilde{\Omega}^*, \mathscr{C}^*_{sing}: \Manifolds \longrightarrow \SymMonCat_\mathrm{str}.
\]
The transformation $R$ along with the property $R\circ a = d$ amounts to a natural transformation
\[
	R: \tilde{E}^* \rightarrow \tilde{\Omega}^*.
\]

\begin{proposition}
Let $\hat{E}^*$ and $\hat{E}'^*$ be two differential extensions of $E^*$. For each manifold $M$, there exists an equivalence of the corresponding symmetric monoidal groupoids
\[
	\Phi_M: \tilde{E}^*(M)\longrightarrow	\tilde{E}'^*(M),
\]
which is natural in $M$.
\end{proposition}
\begin{proof}
Represent $E^*$ by an $\Omega$-spectrum $(E_n, \varepsilon_n)$. Fix a degree $n$. By Proposition 2.3.~in \cite{bunke_schick_uniqueness} there exists a sequence of pointed manifolds $\mathcal{E}_n^{(i)}$, $i\in \bbN$, together with pointed embeddings of submanifolds $\kappa_i: \mathcal{E}_n^{(i)}\hookrightarrow \mathcal{E}_n^{(i+1)}$
as well as $(i-1)$-connected pointed maps $x_i: \mathcal{E}_n^{(i)}\rightarrow E_n$ such that $x_{i+1}\circ\kappa_i=x_i$ and such that for any finite dimensional CW-Complex $X\in \CW^+$ the canonical map
\[
	\colim_{i} [X, \mathcal{E}_n^{(i)}] \rightarrow [X,E_n]=E^n(X)
\]
is an isomorphism. Also, one may choose $\hat{u}_i\in \hat{E}^n(\mathcal{E}_n^{(i)})$ with
\begin{align*}
	\kappa_i^* \hat{u}_{i+1} &= \hat{u}_i,	& I(\hat{u}_i)&=[x_i],	&\mathrm{Rham}(R(\hat{u}_i))=ch[x_i].
\end{align*}
Similarly, one may choose $\hat{u}'_i\in \hat{E}'^n(\mathcal{E}_n^{(i)})$. We are now ready to define $\Phi_M$ on objects: For $\hat{v}\in \hat{E}^n(M)$ write $I(\hat{v})=[M\xrightarrow{f} E_n]$ for some $f$ and factor
$f\simeq x_i\circ f_i$ up to homotopy for some $f_i: M\rightarrow \mathcal{E}_n^{(i)}$. Then
\[
	\hat{v}-f_i^*\hat{u}_i=a(\alpha)\text{ for some } \alpha\in \Omega^{n-1}(M;V)_d
\]
and we define
\[
	\Phi_M(\hat{v}):=f_i^*(\hat{u}'_i)+a'(\alpha).
\]
As is shown in \cite{bunke_schick_uniqueness}, $\Phi_M$ is well-defined, compatible with $I,R,a$ and natural in $M$.
We use the identity on $\Omega^{n-1}(M;V)_d$ to view $\Phi_M$ as a functor
\[
	\Phi_M: \tilde{E}^n(M) \rightarrow \tilde{E}'^n(M),
\]
which is clearly full and faithful. It is also essentially surjective: For $\hat{u}'\in \hat{E}'^*(M)$ choose
$\hat{u}\in \hat{E}^*(M)$ with $I(\hat{u})=I'(\hat{u}')$. Then $\Phi_M(\hat{u})-\hat{u}'=a'(\alpha)$
for some $\alpha\in \Omega^{n-1}(M;V)_d$ and then $\hat{u}-a(\alpha)$ maps to $\hat{u}'$.\\

Thus $\Phi_M$ is an equivalence of categories. It will therefore suffice to show that $\Phi_M$ is a symmetric monoidal functor. The deviation of $\Phi_M$ from being additive is a natural transformation
\[
	(\Phi_M\circ +)-(+'\circ(\Phi_M\times \Phi_M)): \hat{E}^n(M)\times \hat{E}^n(M) \rightarrow \hat{E}'^n(M)
\]
that composes with $I$ to zero, so that it factors over a natural transformation
\[
	B: \hat{E}^n(M)\times \hat{E}^n(M)\rightarrow \Omega^{n-1}(M;V)_d/\im(ch).
\]
satisfying
\begin{align*}
	B(\hat{u},\hat{v}+\hat{w})+B(\hat{v},\hat{w})&=B(\hat{u},\hat{v})+B(\hat{u}+\hat{v},\hat{w}),\\
	B(\hat{u},\hat{v})=B(\hat{v},\hat{u}),\\
	B(\hat{u},0)=0=B(0,\hat{u}).
\end{align*}
Together with the identity arrows $0\xrightarrow{0} \Phi_M(0)=0$ the transformation $B$ gives $\Phi_M$ the structure of a symmetric monoidal functor. The naturality of $B$ -- not in $M$ but in the objects of $\hat{E}^n(M)$ follows from the observation that
\[
	\Phi_M(\hat{v}+a(\beta))=\Phi_M(\hat{v})+a(\beta).
\]
\end{proof}

\begin{remark}\label{bunke_schick_interpretation}
Under additional assumptions (rationally even, finiteness conditions on the coefficients)
the results of \cite{bunke_schick_uniqueness} amount to the assertion that one may even pick a \emph{strict} monoidal
equivalence.
\end{remark}

\section{The Refined Chern Character}

Let $E_n$ be an $\Omega$-spectrum with structure maps $\varepsilon_n: \Sigma E_n\rightarrow E_{n+1}$ which are adjoint to homeomorphisms.\\

Then the fundamental groupoid $\Map(X,E_n)=\Pi_1 E_n^X$ has a monoidal structure: there are two canonical addition maps
\[ E_n\times E_n\approx \Omega^2 E_{n-2} \times \Omega^2 E_{n-2} \longrightarrow \Omega^2 E_{n-2}\approx E_n\]
given by ``horizontal'' and ``vertical'' concatenation
which equip $\Map(X,E_n)$ with two weak monoidal structures $\ominus, \overt$. The associators and unitors are given by the familiar homotopies. The identity object is given in both cases by the constant map $I: X\rightarrow E_n$ onto the basepoint. We have an \emph{equality}
\[  (f_1\ominus g_1)\overt (f_2\ominus g_2)= (f_1\overt f_2)\ominus (g_1\overt g_2)  \]
which together with the unitors yield a natural isomorphism $\ominus \Rightarrow \overt$ (set $g_1=f_2=I$).
Also, by an Eckmann-Hilton type argument we obtain a braid $\gamma_{f,g}: f\ominus g\Rightarrow g\ominus f$ that turns $(\Map(X,E_n),\ominus,I, \gamma)$ into a symmetric monoidal category:
\begin{align*}
	f \ominus g &\cong (I\overt f)\ominus(g\overt I) = (I\ominus g) \overt (f\ominus I) \cong g\overt f\\
	&\cong (g\ominus I)\overt(I\ominus f)=(g\overt I)\ominus(I\overt f)\cong g\ominus f.
\end{align*}

The deRham homomorphisms may be interpreted as a strict symmetric monoidal equivalence of categories
\[  \rho_X: \tilde{\Omega}^n(M) \overset{\sim}{\longrightarrow} \mathscr{C}^n_{sing}(M)  \]
which are the components of a natural transformation $\tilde{\Omega}^n\rightarrow \mathscr{C}^n_{sing}$. Every component is an equivalence and we would like to invert $\rho$. To pick inverse equivalences systematically we use bicategories:\\

With monoidal transformations as $2$-arrows, $\SymMonCat$ forms a strict 2-category. We will view $\Manifolds$ as a 2-category with only identity $2$-arrows.
Then $\rho$ is a strict transformation between the two strict functors $\tilde{\Omega}^n, \mathscr{C}^n_{sing}$ of bicategories.

\begin{lemma}
Suppose $\mathscr{C}$ is a category, considered trivially as a bicategory.
Let
\[F,G: \mathscr{C}\longrightarrow \SymMonCat\]
be strict functors of bicategories
and $u: F\Rightarrow G$ a strict transformation which has the property that all of the functors $u_X: FX\rightarrow GX$ are equivalences. Then there exists a weak transformation $v: G\Rightarrow F$ such
that $uv$ and $vu$ are the identity functors up to modifications.
\end{lemma}
\begin{proof} The monoidal functor
$u_X$ comes equipped with a natural transformation $\mu_X: \otimes(u_X\times u_X)\Rightarrow u_X\circ \otimes$.
For any $X\in\mathscr{C}$ pick $v_X: GX\rightarrow FX$, $\varepsilon_X: 1\Rightarrow u_X\circ v_X$, and $\eta_X: v_X\circ u_X\Rightarrow 1$ such that $(u_X,v_X, \varepsilon_X, \eta_X)$ form an adjoint equivalence. Using the unit $\eta_X$ and counit $\varepsilon_X$ of the adjunction we may make $v_X$ symmetric weak monoidal: define the inverse of $\nu_X$ as
\[  v_X\circ \otimes \overset{v(\varepsilon\otimes\varepsilon)}{\Longrightarrow} v_X\circ \otimes\circ (u_Xv_X\times u_Xv_X)
\overset{v \mu (v\times v)}{\Longrightarrow} v_X\circ u_X \circ \otimes\circ(v_X\times v_X)
\overset{\eta(v\times v)}{\Longrightarrow} \otimes\circ(v_X\times v_X)
\]
The unit $\eta_X$ and counit $\varepsilon_X$ then become monoidal transformations.
It remains to check that the monoidal functors $(v_X,\mu_X)$ depend functorially on $X$ up
to coherent 2-cells $v_f: (Ff,\phi)\circ (v_X,\nu_X) \Rightarrow (v_Y,\nu_Y)\circ (Gf,\gamma)$ for $f:X\rightarrow Y$ in $\mathscr{C}$.
The monoidal transformation $v_f$ is defined as the 2-cell
\[\xymatrix{
GX\ar[rr]^1\ar@/_/[rd]_--{v_X}	&   \ar@{=>}[d]^{\varepsilon_X}	&	GX \ar[rr]^{Gf}	&&	GY\ar@{=>}[d]^{\eta_Y}\ar@/^/[rd]^{v_Y}&\\
& FX\ar@/_/[ru]_--{u_X}\ar[rr]_{Ff} && FY\ar@/^/[ru]^----{u_Y}\ar[rr]_{1} && FY
}\]
Note that the middle diagram commutes.
The coherence conditions follow from the zig-zag equations for $\varepsilon_X, \eta_X$ and
are left to the reader.
\end{proof}

\noindent
We may thus find a \emph{weak} transformation, unique up to modifications,
\[
	\tau: \mathscr{C}^n_{sing}\Longrightarrow \tilde{\Omega}^n.
\]

\begin{proposition}
It is possible to lift the Chern character transformation $ch: [M,E_n] \rightarrow H^n(M;V)$
to a weak monoidal functor
\[
	\Map(M,E_n) \rightarrow \scrC^n_{sing}(M),
\]
natural in $M$, i.e. to a strict transformation $\Map(-,E_n)\rightarrow \mathscr{C}^n_{sing}$.
\end{proposition}
\begin{proof}
By Yoneda's Lemma, the Chern character is determined by `fundamental' cohomology classes $[\iota_n]\in H^n(E_n; V)$. Pick a compatible family $\iota_n\in Z^n_{sing}(E_n,pt;V)$ of `fundamental' cocycles implementing the Chern character, where `compatibility' means that $\int_I \varepsilon_{n-1}^* \iota_n=\iota_{n-1}$ for every $n\in \mathbb{Z}$. A proof that one may choose such compatible families may be found in \cite{mein_erster_artikel}, section 2.1. The required functor is defined on objects by $f\mapsto f^*\iota_n$ and on morphisms by $[H]\mapsto \int_I H^*\iota_n$. The latter is welldefined since the morphisms in $\mathscr{C}^n_{sing}$ are $\im(\delta)$-cosets. Different choices of fundamental cocycles yield modifications of this transformation. The proof that this functor is monoidal can be carried out `by hand', by explicitly writing down coboundaries that verify the coherence conditions. This is sketched in the appendix.
\end{proof}

\vspace{1ex}

Composing this transformation with the weak inverse $\tau$ of the deRham transformation, we obtain a weak transformation
\[
	ch: \Map(-,E_n)\Longrightarrow \tilde{\Omega}^n.
\]
More explicitly, we are given monoidal functors
\[  ch_M: \Map(M,E_n)\longrightarrow \tilde{\Omega}^n(M),  \]
i.e.~a functor $ch_M$ together with a natural transformation, denoted
$ch_M^+: ch_M\circ \ominus \rightarrow +\circ (ch_M\times ch_M)$, and a
map $ch_M(\const) \rightarrow 0$ which may be chosen to be zero (this uses that the fundamental cocycles are reduced, i.e.~vanish on the basepoint).
Also, for every smooth map $f:M\rightarrow N$ we have a map
\[  ch_f: E_n^N\longrightarrow \Omega^{n-1}(M;V)_d  \]
with 
\begin{align}
d(ch_f(c))&=f^*ch_N(c)-ch_M(c\circ f)\hspace{1.34cm}\text{ for all }c: N\rightarrow E_n\label{eqn:1}\\
f^*ch_N(H)+ch_f(H_0)&=ch_f(H_1)+ch_M(H\circ(f\times id_I))\text{ for a homotopy }H:H_0\simeq H_1\label{eqn:2}
\end{align}
where $H: N\times I\rightarrow E_n$. The coherence conditions for $M\xrightarrow{f}N\xrightarrow{g}O\overset{c}{\longrightarrow} E_n$ amount to
\begin{align}
ch_{gf}(c)&=ch_f(c\circ g)  + f^* ch_g(c),\label{eqn:3}\\
ch_{id}&=0\label{eqn:4}.
\end{align}
The compatibility condition for the $\iota_n$ ensure that for a map $h: M\rightarrow E_{n-1}$ viewed as a homotopy $H: M\times I\rightarrow E_n$ from $\const$ to $\const$ we have
\[  ch_M(h)=ch_M(H) \in \Omega^{n-1}(M;V)_d.  \]

\section{Construction of Differential Refinements}

Consider the equivalence relation $\sim$ on the set $\mathscr{E}^n(M)=E_n^M\times \Omega^{n-1}(M;V)_d$ given by $(c_1,\omega_1)\sim (c_2,\omega_2)$ iff there exists a homotopy $H: c_1\rightarrow c_2$ with
\[  ch_M(H)=\omega_1-\omega_2  \]
Reflexivity, symmetry, and transitivity of $\sim$ follow from the following functional properties of $ch_M$: that it preserves the identity, inverses of isomorphisms, and composition. We
may now define
\[   \hat{E}^n(M)=\mathscr{E}^n(M)/\sim.  \]
For a smooth map $M\xrightarrow{f}N$ we let
\[ \hat{E}^n(f): \hat{E}^n(N)\longrightarrow \hat{E}^n(M),\;\; [c,\omega]\mapsto [c\circ f, f^*\omega+ch_f(c)]  \]

This is well-defined by (\ref{eqn:2}) and yields a functor $\hat{E}^n: \Manifolds \longrightarrow \Set$ by (\ref{eqn:3}) and (\ref{eqn:4}). The group structure is given by
\begin{align*}
	[c_1,\omega_1] + [c_2, \omega_2] &= [c_1\ominus c_2, \omega_1+\omega_2 + ch_M^+(c_1,c_2)],\\
	0&=[\const,0].
\end{align*}
Define natural transformations
\begin{align*}
\Omega^{n-1}(M;V)_d&\xrightarrow{a} \hat{E}^n(M),\hspace{1.5cm}  a(\omega)=[\const,\omega]\\
\hat{E}^n(M)&\xrightarrow{I} E^n(M),\hspace{1.5cm}  I[c,\omega]=[c]\in [M,E_n]=E^n(M)\\
\hat{E}^n(M)&\xrightarrow{R} \Omega^n(M;V)_{d=0},\hspace{0.5cm}  R[c,\omega]=ch(c)+d\omega
\end{align*}
The map $R$ is well-defined, again by functoriality of $ch_M$.

\begin{proposition}
$\hat{E}^n$ is a differential cohomology theory in the sense of \cite{bunke_schick_uniqueness}. That is, we have a commutative diagram with an exact row:
\[\xymatrix{
E^{n-1}\ar[r]^-{ch}	&	\Omega^{n-1}(-;V)_d\ar[r]^-a\ar[rd]^d	&	\hat{E}^n\ar[r]^I\ar[d]^R	&	E^n\ar[r]\ar[d]^{ch} & 0\\	&	&	\Omega^{n}(-;V)_{d=0}\ar[r]	&	H^n_{dR}(-;V)
}\]
\end{proposition}
\begin{proof}
The commutativity is clear. Now, if $I[c,\omega]=0$ there exists a homotopy $H: c\rightarrow \const$ so that $[c,\omega]=[\const,\omega+ch_M(H)]=a(\omega+ch_M(H))$. Also $I\circ a=0$.
If $a(\omega)=[\const,0]$ then there exists $H:\const\rightarrow \const$ with $ch_M(H)=\omega$.
and by construction $ch_M(H)$ represents $ch(H)$. Conversely, a map $c: M\rightarrow E_{n-1}$ may be viewed as a homotopy $H: \const\simeq \const: M\rightarrow E_n$. Then
$a(ch(h))=[\const, ch(h)]=[\const,0]$, using $H$ and the compatibility of the $\iota_n$.
\end{proof}

\section{Appendix}

The goal of this section is to show that for a fixed $X\in \ManCat$ the lift of the Chern character
\[  ch: \Map(X,E_n) \longrightarrow \mathscr{C}^n_{sing}(X)  \]
is a monoidal functor. This fact depends on the compatibility condition on the fundamental cocycles $\int_{\Delta^1} \varepsilon_{n-1}^*\iota_n = \iota_{n-1}$ which were used in defining the lift $ch$.
Recall that, using $E_n\approx \Omega E_{n-1}$, the category $\Map(X,E_n)$ is isomorphic to the category $\Map(X\times (I,\partial I), (E_{n-1},*))$ which carries an obvious weak monoidal structure.

The main technical step in the proof is to replace $\Map(X,E_n)$ with an isomorphic weak monoidal category $\mathscr{E}$ in which the associators and unitors are defined in such a way that the coherence conditions on $ch$ are readily `witnessed' by certain canonical coboundaries (remember that a morphism in $\mathscr{C}^n_{sing}(X)$ is defined up to addition of coboundaries). We will make reference to the face and degeneracy maps $d_i, s_j$ of the simplicial set $E_{n-1}^{X\times \Delta^\bullet}$).\\

Objects are $Ob(\mathscr{E})=(E_{n-1},*)^{X\times (\Delta^1,\partial \Delta^1)}$, typical
elements of which are denoted by $f,g,h$. Morphisms $f\rightarrow g$ are homotopy classes
of maps $H: X\times \Delta^2\longrightarrow E_{n-1}$ with $d_2 H=f, d_1 H=g, d_0 H=*$. Here, a homotopy between arrows $H_0$ and $H_1$ from $f$ to $g$ is defined to be a map
$G:X\times \Delta^3 \longrightarrow E_{n-1}$ with $d_0 G=H_1, d_1 G=H_0, d_2 G=s_1g, d_3 G=s_1 f$.

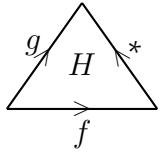
\begin{center}
\begin{tikzpicture}
\draw[thick] (0,0) -- (2,0) node[midway]{\tiny\textbf{>}} node[midway, below] {$f$};
\draw[thick] (0,0) -- (1,1.41) node[midway,sloped]{\tiny\textbf{>}} node[pos=0.4, above]{$g\;$};
\draw[thick] (1,1.41) -- (2,0) node[midway,sloped]{\tiny\textbf{<}} node[pos=0.6, above]{$\;\;*$};
\draw (1,0.6) node{$H$};
\end{tikzpicture}
\end{center}

Just as for simplicial homotopy groups, composition is given by picking arbitrary horn fillers. The identity is $id_f=[s_1 f]$.\\

The monoidal structure is given by the fact that there are \emph{canonical} horn fillers coming from canonical retractions $r_i: \Delta^n\rightarrow \Lambda^n_i,\; n\geq 2$ defined by assigning to $x\in \Delta^n$ the unique intersection point of $\Lambda^n_i$ with the one-dimensional affine subspace stationed at $x$ in direction of the orthogonal complement of $d_i(\Delta^{n-1})$. We will only need $n=2,3$:\\

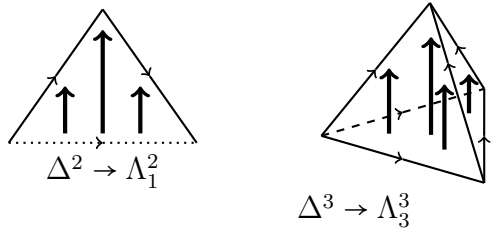
\begin{minipage}[b]{0.5\linewidth}
\flushright

\begin{tikzpicture}[scale=1.25]
\draw [help lines, white] (0,-1) grid (2,1.5);
\draw [thick,->] (0,0) --  (0.5,0.71);
\draw [thick] (0.5,0.71) -- (1,1.42);
\draw [thick,->] (1,1.42)  -- (1.5,0.71);
\draw [thick] (1.5, 0.71) -- (2,0);
\draw [thick,->,dotted] (0,0) -- (1,0);
\draw [thick,dotted] (1,0) -- (2,0);
\draw [->, ultra thick] (1, 0.1) -- (1, 1.2);
\draw [->, ultra thick] (0.6, 0.1) -- (0.6, 0.6);
\draw [->, ultra thick] (1.4, 0.1) -- (1.4, 0.6);
\node at (1,-0.3) {$\Delta^2\rightarrow \Lambda^2_1$};
\end{tikzpicture}
\end{minipage}
\hspace{1cm}
\begin{minipage}[b]{0.5\linewidth}

\tdplotsetmaincoords{60}{30} 
\begin{tikzpicture}[tdplot_main_coords,scale=0.5] 

\draw[->,thick] (0,0,0) -- (2.5,0,0);\draw[thick] (2.5,0,0) -- (5,0,0);
\draw[->, dashed,thick] (0,0,0) -- (1.25, 2.165, 0);\draw[dashed,thick] (1.25, 2.165, 0) -- (2.5, 4.33, 0);

\draw[->,thick] (0,0,0) -- (1.25,  0.72, 2.04);\draw[thick] (1.25,  0.72, 2.04) -- (2.5, 1.443, 4.083);
\draw[->,thick] (5,0,0) -- (3.33,  0.96, 2.72);\draw[thick](3.33,  0.96, 2.72)--(2.5, 1.443, 4.083);

\draw[->,thick] (2.5, 4.33, 0) -- (2.5, 2.89, 2.04);\draw[thick] (2.5, 2.89, 2.04)--(2.5, 1.443, 4.083);

\draw[->,thick] (5,0,0) -- (3.75, 2.16,0);\draw[thick] (3.75,2.16,0) -- (2.5,4.33,0);

\draw[->,ultra thick] (2.5, 1.5,0) -- (2.5,1.5,3);
\draw[->,ultra thick] (1.5, 1,0) -- (1.5,1,2);
\draw[->,ultra thick] (3.2, 1,0) -- (3.2,1,2);
\draw[->,ultra thick] (2.8, 3,0) -- (2.8,3,1.2);

\node at (,0,-2) {$\Delta^3\rightarrow \Lambda^3_3$};
\end{tikzpicture}

\end{minipage}
\vspace{0.1cm}\\
Now, given two objects $f,g \in \mathscr{E}$ we may glue them together to
$X\times \Lambda^2_1 \rightarrow E_{n-1}$ and use $id_X\times r_2$ to canonically
extend to $\sigma(f,g): X\times \Delta^2\rightarrow E_{n-1}$. Then set
\[ f\oplus g= d_1 \sigma(f,g) \]
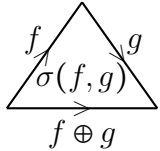
\begin{center}
\begin{tikzpicture}
\draw[thick] (0,0) -- (2,0) node[midway]{\tiny\textbf{>}} node[midway, below] {$f\oplus g$};
\draw[thick] (0,0) -- (1,1.41) node[midway,sloped]{\tiny\textbf{>}} node[pos=0.4, above]{$f\;$};
\draw[thick] (1,1.41) -- (2,0) node[midway,sloped]{\tiny\textbf{>}} node[pos=0.6, above]{$\;\;g$};
\draw (1,0.4) node{$\sigma(f,g)$};
\end{tikzpicture}
\end{center}

The tensor product of morphisms may be defined similarly by canonical extensions.
Given $H_0: f_0 \rightarrow g_0$ and $H_1: f_1\rightarrow g_1$ define $H_0\oplus H_1: f_0\oplus f_1 \rightarrow g_0\oplus g_1$ as follows:

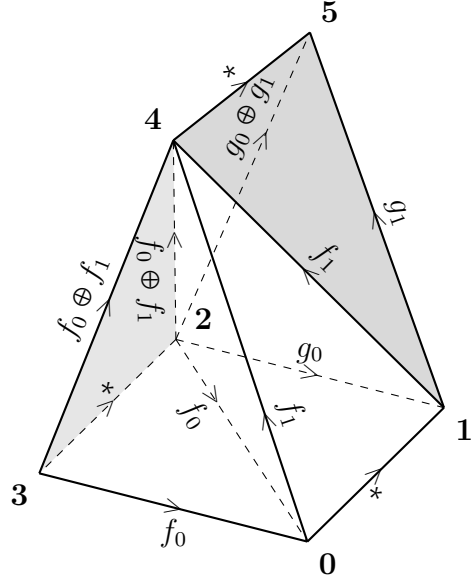
\begin{center}
\tdplotsetmaincoords{60}{27} 
\begin{tikzpicture}[tdplot_main_coords,scale=4]


\filldraw[fill=black!10!white, draw=none]
(0,0,0) -- (0,1,0) -- (0.5,0,1.41);



\filldraw[fill=black!50!white, opacity=0.3, draw=none]
(1,1,0) -- (0.5,1,1.31) -- (0.5,0,1.41);


\draw[thick] (0,0,0) -- (1,0,0)	 node[midway,sloped] {\tiny{\textbf{>}}}	node[midway, below] {$f_0$} node[anchor=north west] {\textbf{0}};
\draw[thick] (1,0,0) -- (1,1,0)	 node[midway,sloped] {\tiny{\textbf{>}}}	node[midway, below] {$*$}  node[anchor=north west] {\textbf{1}};
\draw[dashed] (0,1,0) node[anchor=south west]{$\;$\textbf{2}} -- (0,0,0)	 node[midway,sloped] {\tiny{\textbf{>}}} node[midway, above] {$*$} node[anchor=north east] {$\;$\textbf{3}};
\draw[dashed] (0,1,0) -- (1,1,0)	 node[midway,sloped] {\tiny{\textbf{>}}} node[midway, above] {$g_0$};

\draw[thick] (0,0,0) -- (0.5,0,1.41)	 node[midway,sloped] {\tiny{\textbf{>}}} node[midway, sloped, above] {$f_0\oplus f_1$}  node[anchor=south east] {\textbf{4}};
\draw[thick] (1,0,0) -- (0.5,0,1.41)	 node[pos=0.3,sloped] {\tiny{\textbf{<}}} node[pos=0.3, sloped, above] {$f_1$};
\draw [dashed] (0,1,0) -- (0.5,1,1.31)	 node[pos=0.65,sloped] {\tiny{\textbf{>}}} node[pos=0.7, above, sloped] {$g_0\oplus g_1$} node[anchor=south west] {\textbf{5}};
\draw[thick] (1,1,0) -- (0.5,1,1.31)	 node[midway,sloped] {\tiny{\textbf{<}}} node[midway, above,sloped] {$g_1$};

\draw[thick] (0.5,0,1.41) -- (0.5,1,1.31)	 node[midway,sloped] {\tiny{\textbf{>}}} node[midway,sloped,above] {$*$};

\draw [dashed] (0,1,0) -- (1,0,0) node[pos=0.28, sloped] {\tiny\textbf{>}} node[pos=0.3,sloped,below] {$f_0$};
\draw [dashed] (0,1,0)--(0.5,0,1.41) node[midway,sloped] {\tiny\textbf{<}} node[pos=0.3,sloped, below] {$f_0\oplus f_1$};
\draw [thick](1,1,0) -- (0.5,0,1.41) node[midway,sloped] {\tiny\textbf{<}} node[midway,sloped,above] {$f_1$};
\end{tikzpicture}\end{center}

Canonically fill the horn $\Lambda^3_0$ given by $(\mathbf{0},\mathbf{2},\mathbf{3})=s_0 f_0$, $(\mathbf{2},\mathbf{3},\mathbf{4})=s_0(f_0\oplus f_1)$, $(\mathbf{0},\mathbf{3},\mathbf{4})=\sigma(f_0,f_1)$ to obtain a 3-simplex $\tau=(\mathbf{0},\mathbf{2},\mathbf{3},\mathbf{4})$. Next, canonically fill the horn $\Lambda^3_1$ given by $(\mathbf{0},\mathbf{2},\mathbf{4})=d_0\tau$, $(\mathbf{0},\mathbf{1},\mathbf{2})=H_0$, $(\mathbf{0},\mathbf{1},\mathbf{4})=s_0f_1$ to 
a 3-simplex $\sigma=(\mathbf{0},\mathbf{1},\mathbf{2},\mathbf{4})$. Finally, fill $\Lambda^3_1$ given by
$(\mathbf{1},\mathbf{2},\mathbf{4})=d_1\sigma$, $(\mathbf{1},\mathbf{2},\mathbf{5})=\sigma(g_0,g_1)$, $(\mathbf{1},\mathbf{4},\mathbf{5})=H_1$ to a 3-simplex $\kappa$.
Then $H_0\oplus H_1=d_1\kappa$.\\

The identity object is $*=const: X\times (\Delta^1\times \partial \Delta^1)\rightarrow (E_{n-1},*)$.
The left unitor is $\lambda_f^{-1}=\sigma(f,*): f\rightarrow f\oplus *$. The right unitor is given as the side of an appropriately filled horn $\Lambda^3_1$ given by $(\mathbf{0},\mathbf{1},\mathbf{2})=\sigma(*,f)$, $(\mathbf{1},\mathbf{2},\mathbf{3})=s_1 f$, $(\mathbf{0},\mathbf{1},\mathbf{3})=s_0 f$ as in the next figure. Let $\tau=(\mathbf{0},\mathbf{1},\mathbf{2},\mathbf{3})$ denote the canonical filler. Then the right unitor is $\rho_f=d_1\tau: *\oplus f\rightarrow f$.

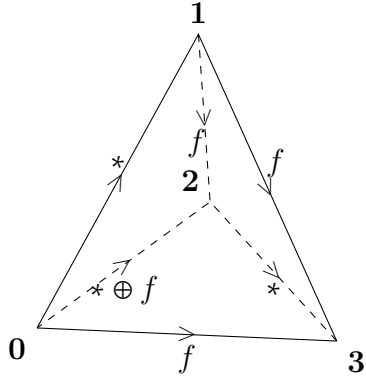
\begin{center}
\tdplotsetmaincoords{60}{5} 
\begin{tikzpicture}[tdplot_main_coords,scale=4]
\draw (0,0,0) node[anchor=north east]{\textbf 0} -- (1,0,0) node[anchor=north west]{\textbf 3} node[midway,sloped]{\tiny\textbf >}node[midway,below]{$f$};
\draw[dashed](1,0,0) -- (0.5,0.886,0) node[anchor=south east]{\textbf 2} node[midway,sloped]{\tiny\textbf >}node[midway,below]{$*$} -- (0,0,0) node[midway,sloped]{\tiny\textbf >} node[midway,below]{$*\oplus f$};
\draw (0,0,0) -- (0.5, 0.44, 0.9) node[anchor=south]{\textbf 1} node[midway,sloped]{\tiny\textbf >}node[midway,above]{$*$} -- (1,0,0) node[midway,sloped] {\tiny\textbf >} node[midway, above]{$\;\;f$};
\draw [dashed] (0.5,0.44,0.9) -- (0.5,0.886,0) node[midway,sloped]{\tiny\textbf >} node[midway,below] {$f\;\;$};
\end{tikzpicture}
\end{center}
The associator is defined as follows:

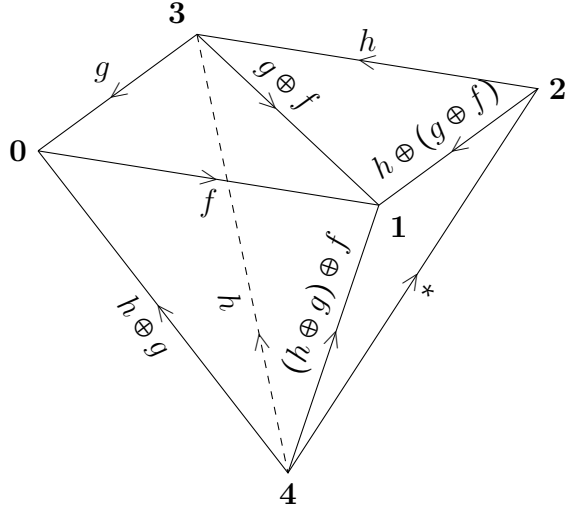
\begin{center}
\tdplotsetmaincoords{70}{25} 
\begin{tikzpicture}[tdplot_main_coords,scale=5]

\draw (0,0,1) node[anchor=east]{\textbf{0}} -- (1,0,1) node[anchor=north west]{\textbf{1}} node[midway,sloped]{\tiny\textbf >}node[midway,below]{$f$} -- (1,1,1) node[anchor=west]{\textbf{2}} node[midway,sloped]{\tiny\textbf <}node[midway,sloped,above]{$h\oplus (g\oplus f)\;\;$} -- (0,1,1) node[anchor=south east]{\textbf{3}} node[midway,sloped]{\tiny\textbf <}node[midway,above]{$h$} -- (0,0,1) node[midway,sloped]{\tiny\textbf <}node[midway,above]{$g\;\;\;\;$};
\draw (0,1,1) -- (1,0,1) node[pos=0.4,sloped]{\tiny\textbf >}node[pos=0.4, above,sloped]{$g\oplus f$};
\draw (0,0,1)--(0.5,0.5,0)node[anchor=north]{\textbf{4}} node[midway,sloped]{\tiny\textbf <}node[midway,sloped,below] {$h\oplus g$}--(1,0,1) node[midway,sloped]{\tiny\textbf >} node[pos=0.6,sloped,above] {$(h\oplus g)\oplus f$};
\draw[dashed] (0,1,1) -- (0.5,0.5,0) node[pos=0.7,sloped]{\tiny\textbf <} node[pos=0.6,sloped,below] {$h$};
\draw (0.5,0.5,0) -- (1,1,1) node[midway,sloped]{\tiny\textbf >} node[midway,sloped, below] {$*$};
\end{tikzpicture}
\end{center}

$(\mathbf{0},\mathbf{1},\mathbf{3})=\sigma(g,f)$, $(\mathbf{0},\mathbf{1},\mathbf{4})=\sigma(h\oplus g,f)$, $(\mathbf{0},\mathbf{3},\mathbf{4})=\sigma(h,g)$ may be filled to a 3-simplex $\tau=(\mathbf{0},\mathbf{1},\mathbf{3},\mathbf{4})$. Next, $(\mathbf{1},\mathbf{3},\mathbf{4})=d_2\tau$, $(\mathbf{1},\mathbf{2},\mathbf{3})=\sigma(h,g\oplus f)$, and $(\mathbf{2},\mathbf{3},\mathbf{4})=s_0 h$ may be filled to another 3-simplex $\kappa=(\mathbf{1},\mathbf{2},\mathbf{3},\mathbf{4})$. Then
\[  \alpha_{h,g,f} = d_2\kappa: (h\oplus g)\oplus f \longrightarrow h\oplus (g\oplus f) \]

Once we show that we have an isomorphism of categories $\Map(X\times (I,\partial I), (E_{n-1},*)) \cong \mathscr{E}$ that maps the tensor product functor, the unitors, and the associators onto each other it will be clear that $\mathscr{E}$ in fact is a monoidal category.
The isomorphism is given by the identity on objects and by assigning to an arrow
$H: f\rightarrow g$ in $\Map(X\times (I,\partial I), (E_{n-1},*))$ the arrow
\[  h: X\times \Delta^2\longrightarrow E_{n-1},\; (s,t)\longmapsto \begin{cases} H(s,t/s) & (s>0),\\  * & (s=0).\end{cases}  \]
where here we view $\Delta^2$ as $\{ (s,t) \in I^2\mid \frac{1-s}{2} \leq t\leq \frac{1+s}{2}  \}\subset I^2$. The functor in the other direction is given on morphisms by a ``vertical'' retraction $I^2\twoheadrightarrow \Delta^2$ mapping $I^2\setminus \Delta^2$ onto $\partial \Delta^2$.\\

The following are a comparison of the two right unitors whose homotopy classes are clearly mapped onto each other:
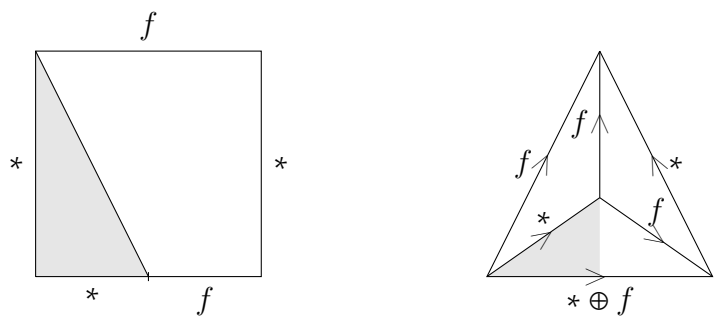
\begin{center}
\begin{tikzpicture}[scale=3]

\filldraw[fill=black!10!white, draw=none]
(0,0) -- (0.5,0) -- (0,1);

\draw (0,0) -- node[pos=0.25,below]{$*$}node[pos=0.75,below]{$f$} (1,0) -- (1,1) node[midway,right]{$*$} -- (0,1)node[midway,above]{$f$} -- (0,0) node[midway,left]{$*$};
\draw (0.5,-0.02) -- (0.5,0.02); \draw (0.5,0) -- (0,1);

\filldraw[fill=black!10!white, draw=none]
(2,0) -- (2.5,0) -- (2.5,0.35);

\draw (2.5,1) -- (2,0)node[midway,sloped]{>}node[midway,left]{$f$} -- (3,0) node[pos=0.48]{>}node[midway,below]{$*\oplus f$} -- (2.5,1) node[midway,sloped]{<} node[midway,right]{$*$};

\draw (2,0) -- (2.5,0.35) node[midway,sloped]{>}node[midway,above]{$*$} -- (3,0) node[midway,sloped]{>} node[midway,above]{$f$};
\draw (2.5,0.35) -- (2.5,1)node[midway,sloped]{>} node[midway,left]{$f$};

\end{tikzpicture}
\end{center}

Now, replacing the category $\Map(X,E_n)$ by $\mathscr{E}$, the functor $ch: \mathscr{E} \rightarrow \mathscr{C}^n_{sing}(X)$ takes $f: X\times \Delta^1\rightarrow E_{n-1}$ to $\int_{\Delta^1} f^*\iota_{n-1}$ and $H: f\rightarrow g$ to $\int_{\Delta^2} H^*\iota_{n-1}$ (this is a morphism from $ch(f)$ to $ch(g)$ on behalf of Stokes' formula and since the fundamental cocycles are reduced). Recall that the tensor product $f\oplus g$ in $\mathscr{E}$ was defined in terms of $\sigma(f,g): X\times \Delta^2 \rightarrow E_{n-1}$.
Then $\int_{\Delta^2} \sigma(f,g):  ch(f\oplus g) \rightarrow ch(f)+ch(g)$ is the transformation 
required by in the definition of a monoidal functor.

The associator in $\mathscr{E}$ was defiend as the boundary of a 3-simplex $\kappa$. The integral $\int_{\Delta^3}$ now `witnesses' the compatibility for the monoidal functor and the associator. The unitors are similarly defined by boundaries of higher simplices $\tau$ and their integrals witness the other coherence conditions.

\bibliography{references}
\nocite{*}

\bibliographystyle{amsalpha}

\end{document}

%% file: preamble.tex
\usepackage[utf8]{inputenc}		
\usepackage[T1]{fontenc}			
\usepackage{hyperref}			

\usepackage{amsmath,amsfonts}	
\usepackage{amsthm,latexsym,amssymb}
\usepackage[mathscr]{euscript}	

\usepackage[all]{xy}

\newtheoremstyle{mytheoremstyle} 
    {0.6cm}                    
    {0.6cm}                    
    {\itshape}                   
    {}                           
    {\bfseries}                   
    {.}                          
    {.5em}                       
    {}  

\newtheoremstyle{mytheoremstarstyle} 
    {0.6cm}                    
    {0.6cm}                    
    {\itshape}                   
    {}                           
    {\bfseries}                   
    {.}                          
    {.5em}                       
    {\thmname{#1}\thmnote{ (#3)}}  
  
\newtheoremstyle{mydefinitionstyle} 
    {0.6cm}                    
    {0.6cm}                    
    {\normalfont}                   
    {}                           
    {\bfseries}                   
    {.}                          
    {.5em}                       
    {}  

\newtheoremstyle{mydefinitionstarstyle} 
    {0.6cm}                    
    {0.6cm}                    
    {\normalfont}                   
    {}                           
    {\bfseries}                   
    {.}                          
    {.5em}                       
    {\thmname{#1}\thmnote{ (#3)}}  
    
\newtheoremstyle{myremarkstyle} 
    {0.6cm}                    
    {0.6cm}                    
    {\normalfont}                   
    {}                           
    {\bfseries}            	       
    {.}                          
    {.5em}                       
    {}  

\newtheoremstyle{myremarkstarstyle} 
    {0.6cm}                    
    {0.6cm}                    
    {\normalfont}                   
    {}                           
    {\bfseries}            	       
    {.}                          
    {.5em}                       
    {\thmname{#1}\thmnote{ (#3)}}  

\theoremstyle{mytheoremstyle}
\newtheorem{theorem}{Theorem}

\newtheorem{lemma}[theorem]{Lemma}
\newtheorem{proposition}[theorem]{Proposition}

\theoremstyle{mytheoremstarstyle}  
\newtheorem{theorem*}{Theorem}
\newtheorem{corollary*}{Corollary}
\newtheorem{lemma*}{Lemma}
\newtheorem{proposition*}{Proposition}

\theoremstyle{mydefinitionstyle}
\theoremstyle{mydefinitionstarstyle}\newtheorem{definition*}{Definition}

\theoremstyle{myremarkstyle}\newtheorem{remark}[theorem]{Remark}
\theoremstyle{myremarkstarstyle}\newtheorem{remark*}{Remark}\newtheorem{example*}{Example}

\def\im{\mathrm{\rm im}}

\def\bbZ{\mathbb{Z}}\def\bbR{\mathbb{R}}\def\bbN{\mathbb{N}}

\def\const{\mathrm{const}}

\def\colim{\mathop{\rm colim}}

\def\CW{\mathbf{CW}}
\def\Set{\mathbf{Set}}

\def\scrC{\mathscr{C}}